\documentclass[a4,12pt]{article} 
\usepackage{amssymb,amsmath,amsthm}
\usepackage{amscd}
\usepackage{hyperref}
\usepackage{epsfig}
\usepackage{amsfonts}
\usepackage{amssymb}
\usepackage{amsmath,enumerate}
\usepackage{euscript}
\usepackage{amscd}

 \newtheorem{theorem}{Theorem}[section]
\newtheorem{remark}{Remark}[section]
\newtheorem{corollary}{Corollary}[section]
  \newtheorem{definition}{Definition}[section]

   \newtheorem{lemma}{Lemma}[section]
    
    \newtheorem{example}{Example}[section]

\title{Weak convergence of the iterations for asymptotically $G$-nonexpansive maps on Banach spaces with a graph}
\author{Asrifa Sultana\footnotemark[1]}
\date{}

\begin{document}

\maketitle
\footnotetext[1]{
Department of Mathematics, Indian Institute of Technology Bhilai, Raipur - 492015, India\\

}
\begin{abstract}
We have derived that on certain Banach spaces having a graph structure $G$, the iterations for asymptotically $G$-nonexpansive map will converge weakly towards a fixed point. This result unifies and extends several theorems on fixed points proved by various authors for class of nonexpansive and asymptotically nonexpansive maps. As an application of this result, we derive that for maps satisfying the
nonexpansive condition locally on special Banach spaces, the successive approximations converge weakly towards a fixed point.
\end{abstract}
{\bf keyword:}
Fixed point, graph structure, uniformly convex Banach space, weak convergence, asymptotically $G$-nonexpansive

\section{Introduction}
In $1967$, Z. Opial \cite {opial} shows that on certain Banach spaces, the successive approximations of nonexpansive maps will converge weakly towards a fixed point.  Later in $1978$, S. C. Bose \cite{bose} extended this mentioned result for class of asymptotically nonexpansive maps. Assume that $(X,||.||)$ is a Banach space and $K\subseteq X$ is non-empty. A map $T:K\rightarrow K$ is known to be asymptotically nonexpansive \cite{goebel} if for any $n \geq 1$ and for all $x,y\in X$,
$$||T^n(x)-T^n(y)||\leq k_n||x-y||,$$
 where the real sequence $\{k_n\}_n$ fulfills $\displaystyle{\lim_{n \to \infty}k_n=1}$. Goebel and Kirk \cite{goebel} originated the notion of such wider class of nonexpansive maps and established the presence of a fixed point for such maps on uniformly convex spaces. In \cite{goebel}, they have also illustrated that asymptotically nonexpansive map need not be nonexpansive with an example. The map $T$ is known as asymptotically regular \cite{bose} when for each $x \in K$, $T^{n}x-T^{n+1}x $ tends to $0$ as $n\rightarrow \infty$. The result for ensuring the convergence of successive iterations for asymptotically nonexpansive maps towards a fixed point due to Bose \cite{bose} is stated below.
 \begin{theorem}\cite{bose}\label{theorembose}
  Suppose that the Banach space $(X,||.||)$ is uniformly convex and endowed with weakly continuous duality map. Assume that $T:K\rightarrow K$ is asymptotically nonexpansive where $K\subseteq X$ is non-empty closed convex and bounded. Then for each $x \in X$, $\{T^{n}(x)\}$ will converge weakly towards a fixed point if $T$ is asymptotically regular.
 \end{theorem}

 In 2008, Jachymski \cite{Jach} first obtained an analogue of the Banach's contraction principle on metric spaces having a graph structure. In this case, the contractive condition on the map need to satisfy only for those pair of elements that are lies in the set of edges of the associate graph. This result also unifies the fixed point theorem for uniformly local contraction proved by Edelstein \cite{edel} and the theorem for monotone contractive mappings on metric spaces endowed with partial ordering due to Ran and Reurings \cite{ran}.

 Our aim in this article is to illustrate that on certain Banach spaces having a graph structure $G$, the successive approximations of asymptotically $G$-nonexpansive map will converge weakly towards a fixed point. In particular,  our result extends and unifies the above mentioned Theorem \ref{theorembose} for asymptotically nonexpansive maps proved by Bose \cite{bose}. As an application, we have deduced that any uniformly locally nonexpansive map on uniformly convex Banach spaces fulfills the characteristic that the successive approximations converge weakly towards a fixed point.  Furthermore, we have concluded that on special Banach spaces endowed with partial ordering, the weak limit of iterations for monotone asymptotically nonexpansive map eventually becomes a fixed point.

\section{Preliminaries}
Suppose that $(X,||.||)$ is a Banach space. The space $X$ is called uniformly convex \cite{clarkson} if for any $0\leq \epsilon<2$, we have a positive $\delta$ so that for any $a,b$ in $X$ with $||a||\leq 1,\,\, ||b||\leq 1$ and $||a-b||\geq \epsilon $, one has $||(a+b)/2||\leq (1-\delta)$.
Note that all finite dimensional Euclidean spaces $\mathbb{R}^{n}$ as well as each Hilbert spaces are typically uniformly convex. Further, the spaces $l_{p}$ and $L_p$ are also uniformly convex \cite{clarkson} for $p\in (1,\infty)$. We now recall the below stated lemma due to Alfuraidan-Khamsi \cite{khamsi} that will be utilized in main section. The lemma is originated by M. Edelstein \cite{edel2}.

\begin{lemma}\cite{khamsi,edel2}\label{lemma1}
Let the Banach space $(X,||.||)$ be uniformly convex and $K\subseteq X$ be non-empty closed and convex. For any bounded sequence $\{x_{n}\}$,
the function $r:K \rightarrow [0,\infty)$ defined as
  $$r(y)=\limsup_{n}||x_{n}-y||,\quad y \in K$$
has a unique minimum $z$ in $K$, that is, $r(z)=\inf\{r(y): y \in K\}=\rho$.   Moreover, if $\{z_{n}\}$ is a minimizing sequence in $K$, that is,
$\lim_{n \rightarrow \infty}r(z_{n})=\rho$, then $\{z_{n}\}$ converges strongly to $z$.
\end{lemma}
  For the given $\{x_{n}\}$, the real number $\rho=\inf\{r(y): y \in K\}$ is known as the asymptotic radius \cite{bose} in the set $K$. The notation $C$ denotes the set $\{p \in K:r(p)=\rho\}$. It is worth to note that $C$ contains single element when $X$ is uniformly convex and in that case $C$ is renowned as the asymptotic center \cite{bose} corresponding to given $\{x_{n}\}$ and the set $K$. We now recall the below stated lemma regarding the asymptotic center.

\begin{lemma}\cite{bose}\label{lemma2}
Let a Banach space $(X,||.||)$ be uniformly convex and endowed with weakly continuous duality map. Let $K\subseteq X$ be a non-empty closed convex and bounded. For any $\{y_{n}\}$ in $K$, the asymptotic center refer to $\{y_{n}\}$ and the set $K$ will be $y^{*}$ if $\{y_{n}\}$ converges weakly to $y^{*}$.
\end{lemma}

\section{Main results}
Let us take a Banach space $(X,||.||)$.  We form a graph $G$ having vertices $V(G)$ same as $X$, whereas $\triangle=\{(y,y):y \in X\}$ contained in the collection of all edges $E(G)$. The notion for asymptotically $G$-nonexpansive map is introduced below.
\begin{definition}
Let $(X,||.||)$ be endowed with the graph $G$ and $K$ be a non-empty subset of $X$. We say $T:K\rightarrow K$ asymptotically G-nonexpansive if
for any $x,y$ in $K$ with $(x,y)\in E(G)$,
$$(T(x),T(y)) \in E(G)\quad and~||T^{n}(x)-T^{n}(y)|| \leq \alpha_{n}||x-y||\quad \forall n,$$
where the real sequence $\{\alpha_{n}\}$ fulfills $\lim_{n \rightarrow \infty}\alpha_{n}=1$.
\end{definition}

\begin{remark}
Each monotone asymptotically nonexpansive map is asymptotically $G$-nonexpansive for the selected graph $G$, in which $E(G)$ coincides with the set $\{(x,y)\in X\times X : x\preceq y\,\,\textrm{or}\,\, y\preceq x\}$.
\end{remark}
\begin{remark}
 Each asymptotically nonexpansive map acts as asymptotically $G$-nonexpansive for the selected graph $G$, in which $E(G)$ matches with the set $\{(x,y): x \in X ~\textrm{and}~ y\in X\}$. We observe that the converse is not happened always and the example given below illustrates that.
\end{remark}

\begin{example}\rm{
Let $X=(l^2,||.||_2)$ and $K\subseteq l^2$ be constructed as $K=\{x \in l^2:||x||_2\leq \frac{1}{2}\}\cup\{(1,0,0,\cdots,0,\cdots)\}$. We form $f:K\rightarrow K$ as
$$f(x_1,x_2,x_3,\cdots)=(0,x_1^2,b_2x_2,b_3x_3,b_4x_4,\cdots),$$
 in which for any $n\geq 2$, $0<b_n<1$  and $\prod_{n=2}^\infty b_n=1$. Let us construct a graph
 $G$, where $V(G)=l^2$ and $E(G)=\{(x,y)\in l^2\times l^2: ||x-y||_2<\frac{1}{2}\}$. Let $x,y$ in $K$ with $(x,y)\in E(G)$. Then both $x,y \in \{x \in l^2:||x||_2\leq \frac{1}{2}\}$ and we can verify that $||f(x)-f(y)||_2\leq ||x-y||_2$, that is, $(f(x),f(y))\in E(G)$. Moreover, $||f^i(x)-f^i(y)||_2\leq ||x+y||_2\prod_{n=2}^i b_n ||x-y||_2\leq \alpha_i||x-y||_2$, where $\alpha_i=\prod_{n=2}^i b_n$. Thus the map $f$ becomes asymptotically $G$-nonexpansive map for the constructed graph. Here $f$ is not asymptotically nonexpansive. Indeed, for $x \in K$ with $||x||_2\leq \frac{1}{2}$ and $y=(1,0,0,\cdots,0,\cdots)$,
 $$||f^i(x)-f^i(y)||_2\leq ||x+y||_2\prod_{n=2}^i b_n ||x-y||_2\leq \alpha_i||x-y||_2,$$
 where $\alpha_i=\frac{3}{2}\prod_{n=2}^i b_n$ and $\lim_{i \rightarrow \infty}\alpha_i\neq 1$.
 }
\end{example}
%
%
%

The main result in this section is stated below for giving the convergence of successive iterations for asymptotically $G$-nonexpansive map weakly towards a fixed point.

  \begin{theorem}\label{Th1_main}
  Let a uniformly convex Banach space $X$ consists with a weakly continuous duality mapping. Let $X$ be endowed with the above defined graph $G$. Suppose $T:K\rightarrow K$ is a asymptotically $G$-nonexpansive map where $K\subseteq X$ is non-empty closed convex and bounded. Assume that
\begin{center}
 for each $\{y_{n}\}$  in $K$, if $(y_{n},y_{n+1})\in E(G)$ for all $n \in \mathbb{N}$ and ~$y_{n_{k}} \rightarrow y^{*}$~weakly,
~then~$(y_{n_{k}},y^{*})\in E(G)~\forall\,k$.
\end{center}
Then for $x$ in $K$, the iteration $\{T^{n}(x)\}_{n}$ will converge towards a fixed point weakly when $T\subseteq E(G)$ and $T$ is asymptotically regular.
\end{theorem}
\begin{proof}
Assume $x \in K$ is a point and define the iterations $\{T^{n}(x)\}_n$. Since $T\subseteq E(G)$, it occurs $(x,T(x))$ lies in $E(G)$.  Because of $(x,T(x))$ lies in $E(G)$ and $T$ is asymptotically $G$-nonexpansive, the pair $(T(x),T^{2}(x))$ lies in  $E(G)$. Similarly, we can visualize that for any $n$, the pair $(T^{n-1}(x),T^{n}(x))$ lies in $E(G)$.
  We will now show that each weak accumulation points of $\{T^{n}(x)\}_{n}$ becomes eventually fixed point for the defined map.

  Suppose that the subsequence $\{T^{n_{k}}(x)\}_{k}$ convergent weakly towards $\bar{x}$ as $k \rightarrow \infty$.
Since $\{T^{n}(x)\}_{n}$ having the property that the pair $(T^{n}(x),T^{n+1}(x))$ lies in $E(G)$ for each $n$ and $\{T^{n_{k}}(x)\}$ approaches towards $\bar{x}$ weakly, it appears $(T^{n_{k}}(x), \bar{x}) \in E(G)$ for any $k$. Let us now define a function $r:K \rightarrow [0,\infty)$ by
  $$r(y)=\limsup_{k}||T^{n_{k}}(x)-y||,\quad y \in K.$$
   Assume that $\rho=\inf\{r(y): y \in K\}$ and $C_{0}=\{z \in K: r(z)=\rho\}$, the asymptotic radius and the asymptotic center, respectively corresponding to $\{T^{n_{k}}(x)\}$ in the set $K$.
  Then according to Lemma \ref{lemma1} and Lemma \ref{lemma2}, we can see that $\bar{x} \in C_{0}$.
For every $p \geq 1$,
   \begin{eqnarray*}
  r(T^{p}(\bar{x}))&=&\limsup_{k}||T^{n_{k}}(x)-T^{p}(\bar{x})||\\
  &\leq &\limsup_{k}||T^{n_{k}+p}(x)-T^{p}(\bar{x})||\quad [\textrm{$\because~T$~is~asymptotic~regular}].
  \end{eqnarray*}
  Since $T$~is~asymptotically $G$-nonexpansive~and~$(T^{n_{k}}(x),\bar{x})\in E(G)~\forall \,k$, we get from the preceding inequality that
  \begin{eqnarray*}
 r(T^{p}(\bar{x}))&\leq & k_{p} \limsup_{k}||T^{n_{k}}(x)-\bar{x}||\\
   &\leq & k_{p} r(\bar{x}).
  \end{eqnarray*}
  For any $p \in \mathbb{N}$, it appears
  $$r(\bar{x})\leq r(T^{p}(\bar{x})) \leq k_{p} r(\bar{x}),$$
  considering that $T^{p}(\bar{x}) \in K$ for each $p$. Since $\lim_{p \rightarrow \infty}k_{p}=1$, we get from the preceding inequality
  $$r(\bar{x})\leq \lim_{p \rightarrow \infty}r(T^{p}(\bar{x})) \leq r(\bar{x}).$$
  Therefore $\lim_{p \rightarrow \infty}r(T^{p}(\bar{x})) =r(\bar{x})=\rho$, that is, $\{T^{p}(\bar{x})\}_p$ is a minimizing sequence of the function $r$. According to Lemma \ref{lemma1}, $\{T^{p}(\bar{x})\}_p$
  converge towards $\bar{x}$ strongly.

   Because of $T\subseteq E(G)$ and $T$ is asymptotically $G$-nonexpansive, the pair $(T^{p}(\bar{x}),T^{p+1}(\bar{x}))$ lies in $E(G)$ where $p\geq 1$. Hence and $\lim_{p \rightarrow \infty}T^{p}(\bar{x})= \bar{x}$ jointly imply that for any $p$, the pair $(T^{p}(\bar{x}), \bar{x})$ lies in the set $E(G)$. Hence
  \begin{eqnarray*}
  ||T(\bar{x})-\bar{x}||
   \leq  \alpha_{1}||\bar{x}-T^{p}(\bar{x})||+||T^{p+1}(\bar{x})-\bar{x}||,
  \end{eqnarray*}
  for each $p \in \mathbb{N}$. Therefore it appears $T(\bar{x})=\bar{x}$.

     We will demonstrate at the last part of the proof that $\bar{x}$ eventually becomes the asymptotic center corresponding to $\{T^{n}(x)\}_n$ in the set $K$.
     Suppose that the asymptotic radius of $\{T^{n}(x)\}_n$ in $K$ is $\rho'$ and apparently we can see $\rho'\geq\rho$.
   Let $\epsilon >0$ be arbitrary. For the chosen $\epsilon$, a positive integer $N$ occurs fulfilling
  $$|| \bar{x}-T^{n_{N}}(x)||\leq \rho+\epsilon/2.$$
  Since $T(\bar{x})=\bar{x}$, it appears for $j \geq 1$,
   $$|| \bar{x}-T^{n_{N}+j}(x)||=||T^{j}(\bar{x})-T^{n_{N}+j}(x)||.$$
   Because of $(T^{n_{N}}(x),\bar{x})$ lies in  $E(G)$ and $T$ is asymptotically $G$-nonexpansive, it occurs from the above equations that for $j \geq 1$,
    $$|| \bar{x}-T^{n_{N}+j}(x)||=||T^{j}(\bar{x})-T^{n_{N}+j}(x)||\leq k_{j}|| \bar{x}-T^{n_{N}}(x)||\leq k_{j}(\rho+\epsilon/2).$$
    Since $k_{j} \rightarrow 1$ as $j \rightarrow \infty$, there is positive integer $J$ satisfying
    $$|| \bar{x}-T^{n_{N}+j}(x)|| \leq k_{j}(\rho+\epsilon/2) \leq \rho+\epsilon \leq \rho'+\epsilon \qquad \forall j \geq J.$$
    Hence it appears $\limsup_{n}||\bar{x}-T^{n}(x)||=\rho'$. Thus the weak accumulation point $\bar{x}$ of $\{T^{n_{k}}(x)\}_{k}$ eventually becomes a fixed point for the defined map and it coincides with the asymptotic center refer to $\{T^{n}(x)\}_n$ and the set $K$. As $\{T^{n_{k}}(x)\}_{k}$ is an arbitrary subsequence, the proof follows.
  \end{proof}
  \begin{remark}
  We can view that this theorem subsumes Theorem \ref{theorembose} proved by Bose \cite{bose} for giving the weak convergence of iterations of asymptotically nonexpansive mappings towards a fixed point if we select a graph $G$ having $E(G)=X \times X$.
  \end{remark}


We are now stating a different version of the preceding Theorem \ref{Th1_main} where we have added the continuity condition on the defined map.
\begin{theorem}\label{new theorem}
Let a uniformly convex Banach space $X$ consists with a weakly continuous duality mapping. Let $X$ be endowed with the above defined graph $G$.
 Let $T:K\rightarrow K$ be a asymptotically $G$-nonexpansive, continuous and asymptotically regular map where $K\subseteq X$ is non-empty closed convex and bounded. Suppose there exist $L \in \mathbb{N}$ and $x_0 \in K$ fulfilling the conditions that
\begin{enumerate}[(a)]
\item $T(x_0) \in [x_{0}]_{G}^{L}$; \label{it 1 Th 1}
\item for the sequence $\{T^{n}(x_0)\}$ in $K$, if $T^{n}(x_0)\in [T^{n-1}(x_0)]_{G}^{L}~\forall\,n \in \mathbb{N}$ and $T^{n_{k}}(x_0) \rightarrow y^{*}$~weakly, ~then~$(T^{n_{k}}(x_0),y^{*})\in E(G)~\forall \,k$. \label{it 2 Th 1}
\end{enumerate}
Then $\{T^{n}(x_0)\}_{n}$ will converge weakly towards a point $x^{*}$ where $x^{*}=T(x^{*})$.
\end{theorem}
\begin{proof}
The proof will be complete if we are able to prove that every weak cluster points of $\{T^{n}(x_0)\}_n$ becomes eventually a fixed point for the defined map. Suppose $T^{n_k}(x_0) \rightarrow x^{*}$ weakly as $k \rightarrow \infty$. Now, following the same lines as used in the proof (first part) of Theorem \ref{Th1_main}, it occurs that $\{T^{p}(x^{*})\}_p$ converges strongly to $x^{*}$. Hence $T^{p+1}(x^{*}) \rightarrow T(x^{*})$ because of $T$ is continuous. Thus it appears $x^{*}=T(x^{*})$. It eventually can be shown by using the similar lines as in the last part of Theorem \ref{Th1_main} that this point $x^{*}$ coincides with the asymptotic center refer to $\{T^{n}(x_0)\}_n$ and the set $K$. As $\{T^{n_k}(x_0)\}_{k}$ is an arbitrary subsequence, the proof follows.
\end{proof}

The preceding Theorem \ref{new theorem} yields that the weak limit of iterations for monotone asymptotically nonexpansive map eventually becomes a fixed point for the defined map on Banach spaces consisting a partial ordering.
 \begin{corollary}
  Let a uniformly convex Banach space $(X,||,||,\preccurlyeq)$ be consist with a weakly continuous duality map where $\preccurlyeq$ is a partial ordering on $X$.  Let $T:K\rightarrow K$ be monotone asymptotically nonexpansive, continuous and asymptotically regular map where $K\subseteq X$ is non-empty closed convex and bounded. Let $x_{0} \in K$ be such that $x_{0}\preccurlyeq T(x_0)$ and
\begin{center}
 for the sequence $\{T^{n}(x_0)\}$, if $T^{n}(x_0)\preccurlyeq T^{n+1}(x_0)$ $\forall \,n \in \mathbb{N}$ and $T^{n_{k}}(x_0) \rightarrow x^{*}$~weakly,
~then~$T^{n_{k}}(x_0)\preccurlyeq x^{*}~\forall\,k$.
\end{center}
Then $\{T^{n}(x_0)\}_{n}$ will converge weakly towards a point $x^{*}$ where $x^{*}=T(x^{*})$.
\end{corollary}
\section{Applications to locally nonexpansive maps}
We are now presenting the stated below result for maps fulfilling the nonexpansive condition uniformly locally by applying preceding Theorem \ref{new theorem} for asymptotically $G$-nonexpansive maps. Edelstein \cite{E5} initiated the concept of uniformly local contractions and derived a significant theorem on fixed points for local contractions. A unique proof for that result using graph structure has been provided in recent times by Jachymski \cite{Jach}.
 \begin{theorem}
 Let a uniformly convex Banach space $(X,||.||)$ consists with a weakly continuous duality map. Suppose $K \subseteq X$ is a non-empty closed convex bounded set and $T:K\rightarrow K$ is a asymptotically regular map so that for some fixed $\varepsilon>0$,
\begin{equation}
||T(x)-T(y)||\leq ||x-y||\quad \textrm{for}~x,y\in K~\textrm{with}~||x-y||<\varepsilon , \label{eq new corollary}
\end{equation}
 Then for $x_{0}$ in $K$, the iteration $\{T^{n}(x_0)\}_{n}$ will converge towards a fixed point weakly when the asymptotic radius $\rho(T^{n}(x_0))<\varepsilon$.
\end{theorem}
\begin{proof}
We first take a graph $G$ having vertices $V(G)$ same as $X$, whereas $E(G)$ contains all $(x,y)$ in $X \times X$ satisfying $||x-y||<\varepsilon$. If we take any two arbitrary points $x,y \in X$ having $(x,y) \in E(G)$, that is, $||x-y||<\varepsilon$. Hence and by (\ref{eq new corollary}), it appears that $||T(x)-T(y)||\leq||x-y||< \varepsilon$, that is,  $(Tx,Ty) \in E(G)$. Hence $T$ becomes asymptotically $G$-nonexpansive for the defined graph. Again, we can check easily that the map becomes continuous. Thus $T$ is continuous, asymptotically $G$-nonexpansive and asymptotically regular mapping. Let $x_{0}$ in $K$ is an element having $\rho(T^{n}(x_0))<\varepsilon$. Since $K\subseteq X$ is convex and $x_0,\,T(x_0)$ belong to the set  $K$, we can find a sequence $(y_i)_{i=0}^{L}$ (for some $L\in \mathbb{N}$) in such a way that $y_0=x_0$, $y_L=T(x_0)$ and $d(y_{i-1},y_{i})<\varepsilon$ for $i=1,\cdots,L$. Therefore it appears that $T(x_0) \in [x_{0}]_{G}^{L}$ and the condition (\ref{it 1 Th 1}) of Theorem \ref{new theorem} holds. 
 As $T$ satisfies the locally nonexpansive condition (\ref{eq new corollary}) and $d(y_{i-1},y_{i})<\varepsilon$ for $i=1,\cdots,L$, it appears that there is a path $(T^{n}(y_i))_{i=0}^{L}$ in $G$ of length $L$ from $T^{n}(x_{0})$ to $T^{n+1}(x_{0})$ for each $n$. Thus $\{T^{n}(x_{0})\}_n$ fulfills the property that $T^{n}(x_{0}) \in[T^{n-1}(x_{0})]_{G}^{L}$~for $n \in \mathbb{N}$. Let us suppose that $T^{n_k}(x_{0}) \rightarrow x^{*}$ weakly. According to Lemma \ref{lemma2}, we see that  $x^{*}$ is the asymptotic center of $\{T^{n_k}(x_{0})\}$ in $K$. This together with $\rho(T^{n}(x_0))<\varepsilon$ imply that $d(T^{n_k}(x_{0}), x^{*}) <\varepsilon$ for all $k$. Thus $(T^{n_k}(x_{0}), x^{*}) \in E(G)$ for each $k$ and hence the condition (\ref{it 2 Th 1}) of Theorem \ref{new theorem} fulfills. Therefore the proof follows according to Theorem \ref{new theorem}.
\end{proof}









\end{document}